\documentclass[12pt]{amsart}
\usepackage[utf8]{inputenc}
\usepackage{amsmath}
\usepackage{amssymb}
\usepackage{amsthm}

\usepackage[all]{xy}
\usepackage{graphicx}
\usepackage{xcolor}
\usepackage{subcaption}
\usepackage{pstricks}
\usepackage{pst-plot}
\usepackage{cancel}
\usepackage{multirow}
\usepackage{siunitx}

\usepackage{xparse}
\usepackage{float}
\usepackage{comment}
\usepackage{blindtext}
\usepackage[a4paper, total={6in, 8in}]{geometry}

\usepackage{tikz-cd}
\usepackage{xfrac}
\usepackage{xparse}
\usepackage{faktor}
\usepackage[thinlines]{easytable}
\newcommand{\CP}{\mathbb{C}P}
\newcommand{\C}{\mathbb{C}}
\newcommand{\R}{\mathbb{R}}
\newcommand{\T}{\mathbb{T}}
\newcommand{\Z}{\mathbb{Z}}

\DeclareMathOperator{\std}{std}
\newcommand{\om}{\omega}

\DeclareMathOperator{\id}{id}
\newtheorem{thm}{Theorem}[section]
\newtheorem{prop}[thm]{Proposition}

\newtheorem{defn}[thm]{Definition}
\newtheorem{lemma}[thm]{Lemma}

\newtheorem{pblm}{Problem}
\title{Almost toric fibrations on symplectic blow ups}
\author{Pranav Chakravarthy and Yoel Groman}
\begin{document}
\maketitle
\begin{abstract}
   Given a symplectic 4-manifold with an almost toric fibration and a symplectic ball embedding whose image under the moment map is contained in an affine convex set $R$, we produce a symplectomorphism between the almost toric blow-up and the symplectic blow-up which is the identity on the pre-image of the complement of $R$. Furthermore, under a compatibility condition of the ball embedding with the boundary divisor, we show that the symplectomorphism can be chosen to preserve the induced symplectic log canonical divisors. 
\end{abstract}
\section{Introduction}
Given a symplectic ball embedding $\iota: B^4(c) \hookrightarrow M$ into a 4-manifold $(M,\om)$, the \emph{symplectic blow up}  $\tilde{M}$ (as described in \cite{MS}) is a local surgery done by removing the interior of the ball embedding and collapsing the boundary sphere along the Hopf fibration. Assume now $M$ is equipped with an \emph{almost toric fibration} $\mu:M \rightarrow B$. Then there is another local surgery construction called the \emph{almost toric blow up} which produces a new manifold $M'$ with a new ATF $\mu':M'\to B'$. It is obtained by removing the pre-image of an appropriate triangle $T(c)\subset B$ and inserting a nodal singularity in its place. A reference is \cite{Symington}. For the convenience of the reader, we review the definitions below. Both operations are surgeries which modify the second cohomology lattice in the same way. 

In this note, we address the question of how these two operations are related. We start by posing the global problem.
\begin{pblm}
    Construct a symplectomorphism $\phi:M'\to\tilde{M}.$
\end{pblm}

When $M$ is a closed rational manifold, such a symplectomorphism is guaranteed to exist by classification results for minimal symplectic 4-manifolds. We recall a particular case of this argument below. The argument employs global methods. In particular, it does not apply to non-rational or non-compact almost toric manifolds. Note that a symplectic sphere in any symplectic $4$-manifold has a normal neighborhood in which it is the part of the boundary of a toric variety. Thus the question of relating the two surgeries is relevant beyond rational manifolds and beyond those which admit a global ATF. For a classification of the latter see  \cite{Leung-Symington}. 

This leads us to pose a more local version of the problem.

\begin{pblm}
    Construct a symplectomorphism $\phi:M'\to\tilde{M}$ which is the identity outside of an appropriately small open set containing $T(c)$ and $\iota(B^4(c))$.
\end{pblm}

Let us refine this question a bit. Observe that symplectic blow up differs from complex blow up in that it does not give rise to a proper transform on divisors. On the other hand, almost toric blow up is somewhat akin to complex blow up in that an ATF has an associated boundary divisor $\mu^{-1}(\partial B)$. The almost toric blow up gives rise to the proper transform - the boundary divisor of $\mu':M'\to B$. Thus, almost toric blow up is properly thought of as an operation on \emph{symplectic log Calabi-Yau pairs.}

Correspondingly for a symplectic log Calabi-Yau pair $(X,D)$ we consider a compatibility condition between a ball embedding and $D$ so that the symplectic blow up gives rise to a new log Calabi-Yau pair $(\tilde{M},\tilde{D}).$ See Definition~\ref{Defn:Compatibleball}.

We can now state our main Theorem. 

\begin{thm}\label{tmMain}
    Let $R\subset B$ be an affine convex subset with an edge $e$ contained in $\partial_i B$, the $i$th edge of the boundary. Let $T(c)\subset R$ have $e$ as one of its edges. Let $R^o$ be the relative interior of $R$ and let $\iota:B^4(c)\to\mu^{-1}(R^o)$ be a symplectic ball embedding contained in  $\mu^{-1}(R^o)$. Then there is a symplectomorphism $\phi:M'\to\tilde{M}$ which is the identity outside of $\mu^{-1}(R)$. Furthermore, if we assume that the ball embedding is compatible with the divisor $D_i$ then $\phi$ can be chosen to be a symplectomorphism of log Calabi-Yau pairs. 
\end{thm}   

\section*{Acknowledgements}
YG was supported by ISF Grant No. 3605/24. 
PC is thankful to Universit\'e Libre de Bruxelles for support during the completion of this project. This work was supported by the FWO and the FNRS via EOS project 40007524. PC is grateful to Jie Min and Margaret Symington for discussions and comments on the document.  
\section{Symplectic and Almost-Toric Blow-up}

\subsection{Definitions}
\subsubsection{Symplectic blow-up}
Given a symplectic ball embedding $\iota: B^4(c) \hookrightarrow M$ of a closed ball of capacity $c$ into a 4-manifold $(M,\om)$, the \emph{symplectic blow-up} $\tilde{M}$ is constructed as follows. We remove the interior of the embedded ball $\iota(\mathrm{Int}(B^4(c)))$ and consider the boundary 3-sphere $\iota(\partial B^4(c)) \cong S^3$. The symplectic blow-up is obtained by collapsing this boundary sphere along the Hopf fibration $S^3 \to S^2 \cong \CP^1$, where each fiber $S^1$ of the Hopf fibration is collapsed to a point. This procedure creates an exceptional sphere $E \cong \CP^1$ in the blown-up manifold $\tilde{M}$. The symplectic area of this exceptional sphere is $\pi c^2$, where $c$ is the capacity of the original ball embedding.
\subsubsection{Almost toric blow up}
\begin{defn}[Almost toric blow-up]\label{Defn:ATFblowup}(See Section 5.4 in \cite{Symington}. See also Section 9.1 in \cite{Evans_2023} for more details)
Let $(M,\om)$ be a symplectic 4-manifold with an almost toric fibration $\mu: M \rightarrow \R^2$ such that the 1-stratum of
the image $B:=\mu(M)$ is non-empty. Then the almost toric blow up/nodal blow up gives us a symplectic manifold $({M'}, {\om}')$ with an almost toric fibration whose base diagram is constructed as follows:
\begin{enumerate}
    \item Choose a base diagram
for $B$ such that for some $c, \epsilon > 0 $ the set $\{(p_1, p_2)~|~ p_1 > -\epsilon, p_2 \geq 0, p_1 + p_2 < c + \epsilon\}$ with
the boundary components marked by heavy lines represent a fibered neighborhood of a
point in the 1-stratum. Remove the triangle with vertices $(0, 0), (c, 0), (0, c)$. 
 \item Connect the two resulting vertices of the base with a pair of dotted lines (As in fig~\ref{fig:ATFblowup})
\end{enumerate}
We say the almost toric blow up \textit{is of size} $c$ and denote the corresponding triangle by $T(c)$. Note that the symplectic area of the exceptional divisor under this surgery is $\pi c^2$.

\end{defn}

{
 \tikzset{every picture/.style={line width=0.75pt}} 
 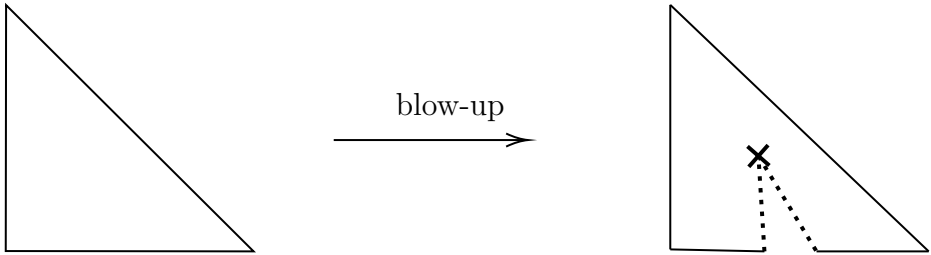
\begin{figure}[H]
    \centering

 \begin{tikzpicture}
 [x=0.75pt,y=0.75pt,yscale=-1,xscale=1] 
 Straight Lines [id:da2462932980075917] \draw (252.5,86) -- (347.5,86) ; \draw [shift={(349.5,86)}, rotate = 180] [color={rgb, 255:red, 0; green, 0; blue, 0 } ][line width=0.75] (10.93,-3.29) .. controls (6.95,-1.4) and (3.31,-0.3) .. (0,0) .. controls (3.31,0.3) and (6.95,1.4) .. (10.93,3.29) ; 
 Straight Lines [id:da6409028868011538] \draw (420.5,141) -- (467.5,142) ; 
 Straight Lines [id:da5723668915133708] \draw [line width=1.5] [dash pattern={on 1.69pt off 2.76pt}] (467.5,142) -- (464.5,94) ; \draw [shift={(464.5,94)}, rotate = 311.42] [color={rgb, 255:red, 0; green, 0; blue, 0 } ][line width=1.5] (-7.27,0) -- (7.27,0)(0,7.27) -- (0,-7.27) ; 
 Straight Lines [id:da24002424902975872] \draw [line width=1.5] [dash pattern={on 1.69pt off 2.76pt}] (464.5,94) -- (493.5,142) ; 
 Straight Lines [id:da3880169037789113] \draw (493.5,142) -- (549.5,142) ; 
 Straight Lines [id:da8878081741076976] \draw (420.5,18) -- (420.5,141) ; 
 Shape: Right Triangle [id:dp9291770470651581] \draw (88.96,18.12) -- (212.5,142) -- (88.79,141.83) -- cycle ; 
 Straight Lines [id:da5717032473053875] \draw (420.5,18) -- (549.5,142) ; 
 Text Node \draw (282,61) node [anchor=north west][inner sep=0.75pt] [align=left] {blow-up};
\end{tikzpicture}
\caption{Almost toric blow-up }
\label{fig:ATFblowup}
 \end{figure}
 }

\subsection{Blow-ups of symplectic divisors}

We want to consider the operations of symplectic blowing up and down as operations on a pair $(M,V)$ of a symplectic manifold and a symplectic divisor $V$. Note that in the symplectic category, the notion of proper transform of a divisor doesn't apriori make sense. Similarly, there is no map from the blow-up to the blow down, and therefore no pushforward of divisors. 

\begin{defn}\label{Defn:Compatibleball}
    We say a divisor $V$ in a symplectic 4-manifold $M$ is \emph{compatible} with a ball embedding $\iota:B(r)\to M$ if for some $\epsilon>0$ the embedding extends to $B(r+\epsilon)$, and $V$ intersects $\iota(B(r+\epsilon))$ as the image of the $z_1$-plane in $\C^2$. We refer to $\iota$ as a \emph{relative ball embedding}. We also say the ball embedding $\iota$ is \emph{$V$-compatible}. 
    \end{defn}
    If $V$ is compatible with the ball embedding $\iota$, we define the symplectic proper transform $\widetilde{V}$ of $V$ in the standard way by looking at the closure of the preimage under the blow down map. We define an analogous notion to allow blow-down of divisors. 
    \begin{defn}
    We say $V$ is compatible with a sphere $S$ of self intersection $-1$ if there is a normal neighborhood of $\mathcal{N}(S)$ of $S$ and a symplectomorphism $\iota$ of a torus invariant neighborhood $U\subset \mathcal{O}(-1)$ of the zero section of the anti-canonical bundle onto $\mathcal{N}(S)$ sending the zero section to $S$ and the fiber over $0$ to the intersection $\iota(U)\cap V$.  
\end{defn}
The blow-down is defined in the standard way by removing the exceptional sphere and glueing in the standard ball in a $V$-compatible way.

\subsection{Symplectic log Calabi Yau pairs}
Recall that a divisor $D$ is said to be a \emph{symplectic log Calabi-Yau divisor} in a symplectic manifold $(X,\om)$, if 
each component of $D$ is a symplectic submanifold that intersects positively transversely with the other components and the homology class of $D$ is Poincar\'e dual to the first Chern class  $c_1(TX,\om)$ of $(X,\om)$.

Let us consider $2$-dimensional symplectic log Calabi-Yau pairs $(X,D)$ for $D=D_1+\dots +D_n$ a log Calabi-Yau divisor. We require that the components of $D$ are $\omega$-orthogonal. 
An \emph{equivalence} of such pairs is a symplectomorphism $\psi:(X,D)\to (X',D')$ so that $\psi_*([D_i])=[D'_i]$. In particular, both divisors have the same number of components. 

We will also use the notion of  a \emph{log CY triple} $(X,D,\mu)$ where $\mu:X\to B$ is an ATF with $D=\mu^{-1}(\partial B).$ 

Before proceeding, we observe that for almost toric blow-up and blow down there are canonical notions of proper transform and blow down for the boundary divisors.

\subsection{Symplectic reduction along boundary components}\label{subsection:sympreduction}
Given a Delzant polygon $\Delta \subset \R^2$, we recall the construction of a symplectic toric manifold whose moment map image is the polygon $\Delta$. Consider the manifold with corners $\Delta \times \T^2$, we compactify this manifold along the boundary by collapsing the orbits corresponding to circle subgroup of $\T^2$ generated by the primitive integral orthogonal vector to the annihilator of the edge. Given a 4-manifold $M$ with an ATF $\mu: M \rightarrow B$ and a convex Delzant polygon $\Delta$, perhaps containing nodes in the interior, and whose boundary may partially overlap with the boundary of the ambient ATF, we can compactify along edges that don't overlap the boundary of the original ATF by collapsing $\Delta \times \T^2$ along circles generated by the primitive integral orthogonal vector to the annihilator of the edge as above. This gives us a closed 4-manifold, with an ATF whose base is the same as $\Delta$ and whose interior is naturally identified with the preimage of the interior of $\Delta$ in the original ATF $\mu$. The above construction plays a crucial role in reducing the proof of Theorem~\ref{tmMain}, to the proof for rational manifolds.

\section{Global isomorphism for blow up toric varieties}
As a first step in proving the main theorem, we prove a special case. Let $X$ be a toric 4-manifold with toric boundary $D$. Let $D_i$ be a component of $D$. We will consider the symplectomorphism between the almost toric blow up and the symplectic blow up, first in the case when the ball embedding is $D_i$ compatible, and then in the case when we only assume the ball embedding does not meet any component of $D$ other than possibly $D_i$.

\subsection{The $D$-compatible case}

 Let $(X',D')$ be the almost toric blow up of capacity $c$ of $D_i$, and let $(\tilde{X},\tilde{D})$ be the symplectic blow-up of a ball of capacity $c$ along a compatible ball embedding $\iota:B^4(c) \hookrightarrow (X,\om)$ that intersects only the divisor $\tilde{D}_i$.  In this section, we present the global isomorphism between the log Calabi-Yau pairs $(X',D')$ and $(\tilde{X},\tilde{D})$. 
 
\begin{prop}\label{globalCase}
There is a symplectomorphism of log CY pairs $(X',D')\to (\tilde{X},\tilde{D})$.
\end{prop}

We will need the following theorem. 

\begin{thm}(Symplectic Torelli Theorem, {Theorem 13 in \cite{Li-Min-Ning}})\label{Theorem:Isotopyofdivisors}
 Let $(M_1,\om_1)$ and $(M_2,\om_2)$ be a pair of rational symplectic 4-manifolds. Let $D_1,D_2$ be a pair of symplectic log Calabi-Yau divisors in $M_1$, $M_2$ with components $D_{1,i},D_{2,i}$ respectively. Suppose the components of $D_1$ and $D_2$ are $\omega_1$-orthogonal and $\om_2$-orthogonal respectively. If there is an integral isometry $$
\gamma: H^2(M_1 ; \mathbb{Z}) \rightarrow H^2(M_2 ; \mathbb{Z})
$$ which maps $PD([D_{1,i}])$ to $PD([D_{2,i}])$, and whose real extension $\gamma_R : H^2(M_2;\R)\rightarrow H^2(M_1;\R)$ maps $[\om_1]$ to $[\om_2]$, then $(M_2,\om_2,D_2)$ is isomorphic to $(M_1,\om_1,D_1)$. 
    
\end{thm}

\begin{proof}[Proof of Theorem \ref{globalCase}]
Let $D'_j$ and $\tilde{D}_j$ denote the components of $D'$ and $\tilde{D}$ respectively. Let $[E']$ and $[\tilde{E}]$ represent the fundamental classes of the exceptional divisors in the blow up.
There exists a map \begin{align*}
    g: H^2(X',\Z)\cong H^2(X,\Z)\oplus \Z[E']&\longrightarrow H^2(X,\Z)\oplus \Z[\tilde{E}] \cong H^2(\tilde{X},\Z) \\
    a+t[E'] &\to a+t[\tilde{E}]
\end{align*}  
 As the ball embedding $\iota$ intersects only one component $\tilde{D}_i$, we have that $g([D'_{j}])= [\tilde{D}_{j}]$ for all $j \neq i$ and $[E']=[\tilde{E}]$, and hence $g([D'_i-E']) = [\tilde{D}_i - \tilde{E}]$. Therefore, the induced map preserves the intersection form and furthermore, satisfies $g([\om']) = [\tilde{\om}]$. Thus, the map satisfies all the properties. The components of the divisors $D'$ and $\tilde{D}$ are orthogonal. Indeed, the original toric divisor $D$ satisfies this. But both the almost toric and the $D_i$ compatible symplectic blow-ups preserve orthogonality. The claim now follows from Theorem~\ref{Theorem:Isotopyofdivisors}.
\end{proof}

\subsection{The non-$D$-compatible case}
We now deal with the case when the ball embedding is not compatible with the divisor.  That is, as before, let $(X,D)$ be a closed toric surface with boundary divisor $D$ and let $(X',\omega',D')$ be the almost toric blow up of a component $D_i$ with triangle of capacity $c$. Let $(\tilde{X},\tilde{\omega})$ be the symplectic blow up with respect to a symplectic ball embedding $\iota: B^4(c) \hookrightarrow (X,\om)$ with capacity $c$  with the only requirement being that the ball does not meet any component of $D$ other than possibly $D_i$.  
\begin{prop}\label{Prop:NonCompatibleCase}
    There exists a symplectomorphism 
    $$\phi:\left(X',\sum_{j\neq i}D_j\right) \to \left(\tilde{X},\sum_{j\neq i}D_j\right)$$ such that $\phi(D_j) = D_j$ for $j \neq i$.
\end{prop}

First we recall the following theorem. 
\begin{thm}\label{Thm:Complex&Symplecticequivalence}
    
    Consider a symplectic ball embedding $\iota: B^4(c) \hookrightarrow M$ with capacity $c$. Push forward the standard complex structure on $\C^2$ to the image of $\iota$ and extend this complex structure to an arbitrary  $\om$-compatible almost complex structure. Let $\hat{M}$ denote the complex blow up centered on $\iota(0)$ with respect to this complex structure. Then there exists a symplectic form $\om_c$ on $\hat{M}$ which assigns area $c$ to the exceptional divisor, and a symplectomorphism $\phi: (\hat{M},\om_c) \rightarrow (\tilde{M},\tilde{\om})$ between the complex blow up and the symplectic blow up  $(\tilde{M},\tilde{\om})$ such that the symplectomorphism is the identity outside a small neighbourhood of $\iota(B^4(c))$.
\end{thm}
\begin{proof}
    This is the content of Lemma 7.1.21 (ii) in \cite{MS}.
\end{proof}

Given a chain of log Calabi-Yau divisors $D=\bigcup_j D_j \subset M$ and a non-D-compatible ball embedding into $M$ intersecting only one of the divisors $D_i$, we now use the above lemma to produce a divisor $D'_i$ of the symplectic blow up $\tilde{M}$ such that $ \bigcup_{j\neq i} D_j \cup D'_i$ is a log Calabi-Yau divisor for $\tilde{M}$.

\begin{lemma}\label{Lemma:CompleteLogCan}
Let $M$ be a 4-manifold with an almost toric fibration $\mu: M \rightarrow B$, with Log Calabi-Yau divisor $\mu^{-1}(\partial B) = \bigcup_{i=1}^k D_i$ and a symplectic ball embedding $\iota: B^4(c) \hookrightarrow M$ with capacity $c$ such that it intersects exactly one of the components of $D_i$. We assume without loss of generality that $\iota(0) \in D_i$. Let $(\tilde{M},\tilde{\om})$ denote the symplectic blow up with respect to the ball embedding. Then there exists a symplectic divisor $D'_i \subset \tilde{M}$ such that $\bigcup_{j \neq i} D_j \cup D'_i$ is an orthogonal log Calabi-Yau divisor for $(\tilde{M},\tilde{\om})$.  
\end{lemma}
\begin{proof}
Push forward the standard complex structure on $\C^2$ to the image of $\iota$ and extend this complex structure to an arbitrary $\om$-compatible almost complex structure. Let $\hat{M}$ denote the complex blow up centered on $\iota(0)$ with respect to this complex structure. By Theorem~\ref{Thm:Complex&Symplecticequivalence},there exists a symplectomorphism $\phi: (\hat{M},\om_c) \rightarrow (\tilde{M},\tilde{\om})$ between the complex blow up equipped with an appropriate symplectic form $\om_c$ and the symplectic blow up  $(\tilde{M},\tilde{\om})$ such that the symplectomorphism is the identity outside a small neighbourhood of $\iota(B^4(c))$. Let $\hat{D}_i$ denote the complex proper transform of $D_i$. Then $D':=\phi^{-1}(\hat{D}_i)$ is the required divisor. As $\phi$ is identity near $\bigcup_{j \neq i} D_j$ and $\hat{D}_i$ coincides with $D_i$ near intersection points, we have that $\bigcup_{j \neq i} D_j \cup D'_i$ is an orthogonal log Calabi-Yau divisor for $(\tilde{M},\tilde{\om})$.
\end{proof}

\begin{proof}[Proof of Proposition \ref{Prop:NonCompatibleCase}]
   There is a symplectomorphism $\phi:X\to X$ so that $\phi\circ\iota(0)$ belongs to $D_i$. Namely, we can pick a path $\gamma$ from $\iota(0)$ to $D_i$ such that $\gamma$ does not meet any of the other divisors. We can then construct an appropriate Hamiltonian supported in an arbitrarily small neighborhood of $\gamma$. 
   By Lemma \ref{Lemma:CompleteLogCan} we have a log Calabi-Yau divisor $\bigcup_{j \neq i} D_j \cup D'_i$  for the symplectic blow up of  $\phi\circ\iota$. Since $\phi$ is the identity near $\bigcup_{j \neq i} D_j$, we have that $\bigcup_{j \neq i} D_j \cup D'_i$ is a log Calabi-Yau divisor for $(\tilde{X},\tilde{\omega})$. So by defining an algebraic map similar to the proof of Theorem~\ref{globalCase} and applying the Symplectic Torelli Theorem (Theorem~\ref{Theorem:Isotopyofdivisors}) we have a symplectomorphism of symplectic log Calabi Yau pairs where the components $D_j$ for $j\neq i$ map to the corresponding component $D_j$.  
\end{proof}

\subsection{Adjustment near the boundary}

\begin{prop}\label{lmBULocalCase}Using the notation of Propositions \ref{globalCase} and \ref{Prop:NonCompatibleCase}, let $E=\sum_{j\neq i}D_j$ in $X'$ and $\tilde{X}$ respectively the sum of the non-blown up boundary divisors. Then there is a symplectomorphism of symplectic pairs $(X',E)\to (\tilde{X}, E)$ which is the identity in a neighborhood of $E$. 
\end{prop}
Before proceeding with the proof, note that both symplectic manifolds $X'$ and $\tilde{X}$ are obtained by surgeries from the same toric variety $X$, with the surgeries supported away from the non-blow-up divisors $E$. Thus there is a notion of an identity map in a neighborhood of $E$.

To prove Proposition \ref{lmBULocalCase}, we need the following two propositions from \cite{Ch-Gro}:

\begin{lemma}[Proposition 3.8 in \cite{Ch-Gro}]\label{Prop:IsotopyToIdentity}
Let $(M,\om)$ be a symplectic manifold and let $S\subset M$ be a symplectic submanifold. Let $\Phi$ be a symplectomorphism of an open neighbourhood of $S$ onto an open neighbourhood of $S$. Suppose $\Phi|_S = \id$ and $d\Phi$ restricted to $TM|_S$ is the identity as well. 
Then $\Phi$ is Hamiltonian isotopic to the identity on a possibly smaller neighbourhood $U'$ of $S$. 
Furthermore, if $\Phi$ is the identity on an open neighbourhood $V$ in $M$ of a closed subset $A\subset S$ then the Hamiltonian can be taken to be $0$ on the union of $S$ with a possibly smaller neighbourhood of $A$. 
\end{lemma}

\begin{lemma}[Proposition 3.9 in \cite{Ch-Gro}]\label{Lemma:MakingNormalIdentity}
Let $(M,\om)$ be a symplectic $4$-manifold and $S\subset M$ a  symplectic sphere. Let $U$ be an open neighbourhood of $S$ and $\phi:U\rightarrow M$ a symplectic embedding whose restriction $\phi|_S$ is the identity. Then there exists a Hamiltonian isotopy $\beta_t:M \rightarrow M$ such that $\beta_t|_S = id$ and on  $d\beta_1 = d\phi$ on $S$. Moreover, if there is a closed subset $A\subset S$ for which $\phi$ is the identity on an open neighborhood of $A$ in $M$, then $\beta$ can be taken to be the identity on some open neighborhood of $A$.
\end{lemma}

\begin{proof}[Proof of Proposition \ref{lmBULocalCase}]

The idea of the proof is similar to proof of Proposition 3.7 in \cite{Ch-Gro} with the added simplification that $E$ is not a closed cycle of divisors allowing us to carry out the first step below.

By Theorem~\ref{globalCase} we have a symplectomorphism of pairs $\phi:(X',E)\to (\tilde{X},E)$. We will modify $\phi$ to be the identity near $E$ under the identification between $X',\tilde{X}$ in a neighborhood of $E$ which is disjoint of the surgery locus.

\textbf{Step 1:} In this step we will modify $\phi$ in an arbitrarily small neighborhood of $E$ so that the restriction to $E$ is the identity. Consider an irreducible sphere component $D_j$ of $E$. As the group of symplectomorphisms of the sphere is connected and since $\pi_1(S^2)=0$, we have that $\phi^{-1}|_{D_j}$ is generated by a Hamiltonian $H_j^t:D_j \rightarrow \R$. Since $\phi$ fixes the intersection points $D_{ij}$, we can assume the Hamiltonian isotopy fixes the intersection points $D_{j,j-1},D_{j,j+1}$ of $D_j$ with the adjacent divisors. We further normalize the Hamiltonian $H^t_j$ so that it is identically $0$ on $D_{j,j-1}$. This is possible to accomplish for all $j$ since $E$ is a chain (but not a cycle) of divisors.

We can extend each of the Hamiltonian $H^t_j$ to a global Hamiltonian symplectomorphism $\Psi_j$ which is supported in an arbitrarily small open neighborhood $V_j$ of $D_j$ such that the restriction to $\Psi|_{D_{j,j-1}}$ is identity. This is accomplished by identifying a neighbourhood of the 0-section in $N(D_j)$ with a neighborhood  $V_j$ of $D_j$ in $X'$. Since the components of $D$ are symplectically orthogonal to one another, we can arrange the identifications so that the $D_{j-1}\cap V_j$,$D_{j+1}\cap V_{j+1}$ are fibers, and the fibration is locally trivial near the intersection points. We pick a bump function $\rho_j$ which is supported in $V_j$, is identically $1$ near $D_j$ and whose Hamiltonian flow preserves the intersections $D_{j,j-1}\cap V_j,D_{j,j+1}\cap V_i$. Denoting by $\pi_j:V_j\to D_j$ the projection under the above identification of $V_j$ with the normal bundle, the symplectomorphism $\Psi_j$ generated by $\rho_jH_j\circ \pi_j$ satisfies the advertised properties. 

Let $\phi':=\Psi_m\circ\dots \Psi_1\circ\phi$, where $m$ is the number of divisors in $E$. Then
\begin{enumerate}
  \item $\phi'$ preserves $D$
    \item $\phi'$ is the identity on $E$. 
    \item $\phi'$ agrees with $\phi$ outside an arbitrarily small neighborhood of $E$.
\end{enumerate}
as required. 
Abusing notation, we denote the above modified symplectomorphism $\phi'$ by $\phi$. In the next steps, we modify $\phi$ to be the identity in a neighbourhood of  $E$.

\textbf{Step 2:}   We now modify $\phi$ to be identity in a neighbourhood of the $0$-strata i.e in a neighbourhood of each $D_{ij}$, while preserving the strata. This is accomplished by the \emph{Alexander trick.}
We have that $\phi$ fixes the $0$-dimensional strata of the divisor $D$. Let $\omega'=\omega_{\std}$ be the standard symplectic form in a small Darboux ball $B_{ij}$ centered at  $D_{ij}$. By orthogonality we can symplectically identify $B_{ij}$ with the standard ball in $\C^2$ so that $D_i,D_j$ map to the coordinate axes. Restricting $\phi^{-1}$ to a smaller ball $V$, we use the Alexander trick to connect $\phi^{-1}|_{V}$ to the identity via symplectomorphisms which preserve the coordinate axes. More specifically, define 
$$\phi^{-1}_t(v):=\begin{cases}
        \frac1{t}\phi^{-1}(tv),&\quad t\neq 0\\
        d\phi^{-1}(0,0),&\quad t=0.
        \end{cases}
        $$
By the previous step, $d\phi^{-1}(0,0)$ is the identity. We thus get a time dependent symplectic vector field $X_t$ with domain $\phi^{-1}_t(V)$, by $p\mapsto \frac{d}{dt}\phi^{-1}_{t_0+t}(p)|_{t=0}$. By simple connectedness, we can find a time dependent Hamiltonian $H_t:\phi^{-1}_t(V)\to \R$ so that $X_t$ is its Hamiltonian vector field. Let $\epsilon>0$ be such that $B_{\epsilon}\subset \phi^{-1}_t(V)$ for all time $t\in[0,1]$. 

Let $f:\R^2\to[0,1]$ be a function which is identically $0$ for $x+y\geq\epsilon$, identically $1$ close enough to the origin, and such that $\frac {\partial f}{\partial x} $ is identically $0$ in a sufficiently small neighborhood of the $y$ axis, and $\frac {\partial f}{\partial y} $ is identically $0$ in a sufficiently small neighborhood of the $x$ axis. Let $g:\C^2\to\R$ be defined by $g(z_1,z_2):=f(|z_1|^2,|z_2|^2)$. Then the Hamiltonian flow of $gH_t$ is globally defined and preserves the axes. Let $\Psi$ denote global symplectomorphism given by the time 1-flow of $gH_t$. Thus $\tilde{\phi}:= \phi \circ\Psi$ is a modification of $\phi$ that satisfies the following conditions as required. 

\begin{enumerate}\label{Properties1to3}
       \item $\tilde{\phi}(D)=D$ componentwise
       \item $\tilde{\phi}|_E$ is the identity
       \item $\tilde{\phi}$ is the identity near the intersection points $D_{ij}$
       \item $\tilde{\phi}$ agrees with $\phi$ away from an arbitrarily small neighborhood of $E$
   \end{enumerate}

 We continue to abuse notation and denote by $\phi$ the modified map from the previous step.

\textbf{Step 3:} We now modify $\phi$ so that it satisfies all the properties of the previous part, and, in addition, the differential $d\phi|_E$ . Namely,  for each $i$, Proposition \ref{Lemma:MakingNormalIdentity}  produces a Hamiltonian symplectomorphism $\beta_i$  so that $\beta^{-1}_i\circ\phi$ is the identity on $D_i$ and has differential equal to identity. Moreover, the proposition says that $\beta_i$ can be taken to be identity on a small neighborhood of $D_{i,i-1}\cup D_{i,i+1}$ where $\phi$ is already identity by the previous step. Thus, defining $\beta=\beta_m\circ\dots\beta_1$ we have that $\beta\circ\phi$ satisfies the desired property. Moreover, we can make the support of $\beta$ arbitrarily close to $E$ by bumping off the Hamiltonian generator. We still denote the result by $\phi$.

\textbf{Step 4:} Finally, we apply Proposition~\ref{Prop:IsotopyToIdentity} to further modify $\phi$ to be the identity near $E$. Namely, for each $i$ the proposition allows us to modify $\phi$ near each $E_i$ to be identity. Moreover, the generating Hamiltonian can be taken to be $0$ near $D_{i,i-1}$ and away from an arbitrarily small neighborhood of $D_i$. So we inductively apply this modification to $D_{i+1}$ while preserving this property for the union $D_1\cup \dots\cup D_{i}$.

The resulting symplectomorphism $\phi$ satisfies the requirement.
\end{proof}

\section{Proof of Theorem~\ref{tmMain}}

\begin{proof}
    Since $R$ is convex, we can find a Delzant polygon $P\subset R\subset B$ such that $\mu(\iota(B^4(c)))\subset P$. Applying Lemma \ref{globalCase} in the $D$-compatible case and Lemma \ref{Prop:NonCompatibleCase} in the non-$D$-compatible case to the symplectic reduction along the boundary of $P$ (as described in Section~\ref{subsection:sympreduction}) for each of the two blow ups, and modifying the resulting symplectomorphism in accordance with Lemma \ref{lmBULocalCase}, we obtain a symplectomorphism between the almost toric and symplectic blow up of $\mu^{-1}(P^o)$ which is the identity near the boundary. It thus extends to a symplectomorphism as desired. 
\end{proof}

\bibliographystyle{amsalpha}
\bibliography{bibliography}

\providecommand{\bysame}{\leavevmode\hbox to3em{\hrulefill}\thinspace}
\providecommand{\MR}{\relax\ifhmode\unskip\space\fi MR }
\providecommand{\MRhref}[2]{%
  \href{http://www.ams.org/mathscinet-getitem?mr=#1}{#2}
}
\providecommand{\href}[2]{#2}
\begin{thebibliography}{LMN25}

\bibitem[CG25]{Ch-Gro}
Pranav {Chakravarthy} and Yoel {Groman}, \emph{{Almost toric fibrations on K3 surfaces}}, arXiv e-prints (2025), arXiv:2502.04304.

\bibitem[Eva23]{Evans_2023}
Jonny Evans, \emph{Lectures on lagrangian torus fibrations}, London Mathematical Society Student Texts, Cambridge University Press, 2023.

\bibitem[LMN25]{Li-Min-Ning}
Tian-Jun Li, Jie Min, and Shengzhen Ning, \emph{Almost toric presentations of symplectic log calabi-yau pairs}, arXiv e-prints (2025), 2303.09964.

\bibitem[LS10]{Leung-Symington}
Naichung~Conan Leung and Margaret Symington, \emph{Almost toric symplectic four-manifolds}, Journal of Symplectic Geometry (2010).

\bibitem[MS17]{MS}
Dusa McDuff and Dietmar Salamon, \emph{Introduction to symplectic topology}, third ed., Oxford Graduate Texts in Mathematics, Oxford University Press, Oxford, 2017. \MR{3674984}

\bibitem[Sym03]{Symington}
Margaret Symington, \emph{Four dimensions from two in symplectic topology}, Topology and geometry of manifolds (editors G Matić, C McCrory), Proc. Sympos. Pure Math. 71, Amer. Math. Soc. (2003) 153 MR2024634 (2003).

\end{thebibliography}
\end{document}